\documentclass[10pt,twoside]{amsart}

\usepackage{amscd}
\usepackage{amssymb} 
\usepackage{verbatim}
\usepackage{amsmath}
\usepackage{amsxtra}
\usepackage[noadjust]{cite}         
\usepackage{epsfig}
\usepackage{xcolor}

\newtheorem{theorem}{Theorem}[section]
\newtheorem{lemma}[theorem]{Lemma}

\theoremstyle{definition}

\numberwithin{equation}{section}

\newcommand{\RZ}{{\mathbb R} \slash {\mathbb Z}}
\newcommand{\bR}{{\mathbb R}}

\newcommand{\bRplus}{{\mathbb R}_{>0}}
\newcommand{\bRgeq}{{\mathbb R}_{\geq 0}}
\newcommand{\altm}{a}
\newcommand{\dH}[1]{\;{\rm d}{\mathcal{H}}^{#1}} % Hausdorff measure
\newcommand{\dd}[1]{\frac{\rm d}{{\rm d}#1}}
\newcommand{\ddt}{\dd{t}}
\newcommand{\drho}{\;{\rm d}\rho}
\newcommand{\ds}{\;{\rm d}s}
\newcommand{\Vh}{\underline{V}^h}
\newcommand{\ratio}{{\mathfrak r}}
\newcommand{\Id}{{\rm Id}}
\newcommand{\dist}{\operatorname{dist}}
\renewcommand{\deg}{{d}}
\newcommand{\sgn}{\operatorname{sgn}}

\begin{document}

\title[A stable scheme for anisotropic curve shortening flow]{%
An unconditionally stable finite element scheme for anisotropic curve 
shortening flow}

\author[K. Deckelnick]{Klaus Deckelnick}
\address{Institut f\"ur Analysis und Numerik,
Otto-von-Guericke-Universit\"at Magdeburg, 39106 Magdeburg, Germany}
\email{klaus.deckelnick@ovgu.de}

\author[R. N\"urnberg]{Robert N\"urnberg}
\address{Dipartimento di Mathematica, Universit\`a di Trento,
38123 Trento, Italy} 
\email{robert.nurnberg@unitn.it}

\thanks{}

\begin{abstract}
Based on a recent novel formulation of parametric anisotropic curve
shortening flow, we analyse a fully discrete numerical method
of this geometric evolution equation. The method uses piecewise linear finite
elements in space and a backward Euler approximation in time.
We establish existence and uniqueness of a discrete solution, as well as
an unconditional stability property.
Some numerical computations confirm the theoretical results and 
demonstrate the practicality of our method.
\end{abstract}



\maketitle

\section{Introduction} % \label{sec:1}

In this paper we study a fully discrete numerical scheme for parametric anisotropic curve shortening flow. This evolution
law arises as a natural gradient flow for the energy
\begin{equation} \label{eq:aniso}
\mathcal E(\Gamma) = 
\int_{\Gamma} \altm(z)\gamma(z,\nu)\dH1(z)
= \int_{\Gamma} \altm\, \gamma(\cdot,\nu) \dH1,
\end{equation}
where $\Gamma$ is a closed curve with unit normal $\nu$ contained in a given 
convex domain $\Omega \subset \bR^2$. 
Furthermore, $a \in C^1(\Omega, \bRplus)$ is a weight function and
$\gamma \in C^0(\Omega \times \bR^2, \bRgeq) \cap 
C^2(\Omega \times (\bR^2 \setminus \{ 0 \}), \bRplus)$
denotes an anisotropy function satisfying
\begin{equation} \label{eq:abshom}
\gamma( z, \lambda  p) = | \lambda | \gamma( z,  p) \quad \text{ for all }  p \in \bR^2,\ \lambda \in \bR,\  z \in \Omega.
\end{equation}
In addition, we assume that $\gamma$
is strictly convex in the sense that for every compact $K \subset \Omega$ there exists $c_K>0$ such that
\begin{equation*} % \label{eq:ass2}
\gamma_{pp}(z,p) q \cdot q \geq c_K | q |^2 \quad \text{ for all } z \in K,\ p,q \in \bR^2 \text{ with } p \cdot q =0, |p|=1.
\end{equation*}
Here, and in what follows,
$\gamma_p$ and $\gamma_{pp}$ denote gradient and Hessian of the 
function $p \mapsto \gamma(\cdot, p)$.
Anisotropic energies
of the form \eqref{eq:aniso} play a role in applications, 
such as materials science, crystal growth, phase transitions and image
processing, and in differential geometry, and we refer to e.g.\ 
\cite{TaylorCH92,Gurtin93,BellettiniP96,ClarenzDR03,Bellettini04a,Giga06,
CasellesKS97,AlfaroGHMS10,glns}
for examples and further details. 
It can be shown, see \cite[Appendix~A]{finsler},
that the first variation of $\mathcal E$ in the direction of 
a smooth vector field $V$ is given by
\begin{equation} \label{eq:firstvar}
d \mathcal E(\Gamma;V)= - \int_{\Gamma} a\, \gamma(\cdot,\nu)
\,\varkappa_\gamma \, V \cdot \nu_\gamma \dH1,
\end{equation}
where 
\[
\nu_\gamma = \frac\nu{\gamma(\cdot,\nu)} ~\text{ and }~
\varkappa_\gamma = \varkappa \gamma_{pp}(\cdot,\nu) \tau \cdot \tau 
- \sum_{i=1}^2 \gamma_{p_i z_i}(\cdot,\nu) 
- \frac{\nabla \altm}{\altm} \cdot \gamma_p(\cdot,\nu) \quad \text{ on } \Gamma
\]
denote the anisotropic normal and the anisotropic curvature of $\Gamma$,
respectively, with
$\tau$ and $\varkappa$ the tangent and curvature of $\Gamma$. 

In view of \eqref{eq:firstvar}, a natural gradient flow for the energy $\mathcal E$ evolves a family of closed curves $(\Gamma(t))_{t \in [0,T]}$ according
to 
\begin{equation}  \label{eq:anisolaw}
\mathcal V_\gamma = \varkappa_\gamma \quad \text{ on } \Gamma(t),
\end{equation}
where $\mathcal V_\gamma = \frac{1}{\gamma(\cdot,\nu)} \mathcal V$ and
$\mathcal V$ is the normal velocity of $\Gamma(t)$. 
We remark that solutions of \eqref{eq:anisolaw} satisfy the energy relation
\begin{equation*} %\label{eq:anigradflow}
\ddt \int_{\Gamma(t)} a \, \gamma(\cdot,\nu) \dH1 + \int_{\Gamma(t)}
 | \mathcal V_\gamma |^2 \, a \,  \gamma(\cdot,\nu) \dH1  = 0.
\end{equation*}
Note that in the isotropic case, $a(z) = 1$ and $\gamma(z,p) = |p|$,
the flow \eqref{eq:anisolaw} collapses to the well--known curve shortening flow $\mathcal V =\varkappa$. 
The isotropic curve shortening flow and its higher dimensional analogue, the mean curvature flow, have been
studied extensively both analytically and
numerically over the last few decades, and we refer to the works
\cite{Ecker04,DeckelnickDE05,Mantegazza11,bgnreview} for more details. 

In the spatially homogeneous case, $a(z) = 1$ and $\gamma(z,p) = \gamma_0(p)$, the
flow \eqref{eq:anisolaw} reduces to the classical anisotropic curve shortening
flow 
\begin{equation} \label{eq:classicalacsf}
\frac{1}{\gamma_0(\nu)} \mathcal V = \varkappa_{\gamma_0},
\end{equation}
where $\varkappa_{\gamma_0} = \varkappa \gamma_{0}''(\nu) \tau \cdot \tau$
denotes the usual anisotropic curvature. 
An example for a nonconstant function $a$ is given by the geodesic curvature
flow in a Riemannian manifold, see \S\ref{sec:geo} and 
\cite[Appendix~B]{finsler} for details.

In this paper we focus on a parametric description of the evolving curves, i.e. 
$\Gamma(t) = x(I,t)$ with 
$x : I \times [0,T] \ni (\rho,t) \mapsto x(\rho,t) \in \bR^2$ 
and $I = \RZ$. Hence the evolution law \eqref{eq:anisolaw} translates into
\begin{equation} \label{eq:acsfI}
\frac{1}{\gamma(x,\nu)} x_t \cdot \nu = \varkappa_\gamma,
\end{equation}
where $\nu = \tau^\perp=(\frac{x_\rho}{|x_\rho|})^\perp$ and
$p^\perp = \binom{p_1}{p_2}^\perp = \binom{-p_2}{p_1}$ denotes
an anti-clockwise rotation of $p$ by $\frac{\pi}{2}$.
Note that
here, and from now on, we think of $\tau$, $\nu$, $\varkappa$ and 
$\varkappa_\gamma$ as being defined on $I \times [0,T]$. 

In order to obtain solutions of \eqref{eq:acsfI}, frequently the %system of 
partial differential equation (PDE) given by
\begin{equation} \label{eq:Dziukani}
\frac{1}{\gamma(x,\nu)} x_t = \varkappa_\gamma \nu
\end{equation}
is solved, with the initial condition $x(\cdot,0)=x_0$, where $x_0$ is a 
parameterization of the initial curve $\Gamma_0$.
Since the right hand side of \eqref{eq:Dziukani} is a geometric invariant, the 
above PDE appears to be a natural choice. 
Let us focus for a moment on the isotropic case $a(z) = 1$ and $\gamma(z,p) = |p|$. Then the
system \eqref{eq:Dziukani} takes the form
\begin{equation} \label{eq:Dziuk}
x_t = \varkappa\nu = \frac1{| x_\rho|} \left( \frac{ x_\rho}{| x_\rho|} \right)_\rho= \frac1{| x_\rho|^2} \left[x_{\rho\rho} 
- ( x_{\rho\rho} \cdot \tau)\, \tau \right],
\end{equation}
so that the underlying PDE  is only weakly parabolic, causing difficulties 
for the numerical analysis. We refer to Dziuk's seminal paper \cite{Dziuk94} 
for the details. A simple remedy is to apply the so-called DeTurck trick, 
and to consider, in place of \eqref{eq:Dziuk},
the strictly parabolic PDE
\begin{equation} \label{eq:DD95}
x_t = \frac{x_{\rho\rho}}{| x_\rho |^2} ,
\end{equation}
whose solutions clearly still satisfy $x_t \cdot \nu = \varkappa$.  
This formulation was proposed
and analysed for the first time in \cite{DeckelnickD95}, see also
\cite{ElliottF17} for a possible generalization.

Extending the DeTurck trick \eqref{eq:DD95} for the isotropic flow
to the anisotropic evolution equation \eqref{eq:acsfI} is highly
nontrivial. However, the main idea is the same: derive a strictly 
parabolic PDE whose solutions satisfy \eqref{eq:acsfI}. In this way,
a uniquely defined tangential velocity is prescribed together with
the normal velocity \eqref{eq:acsfI}, yielding a unique parameterization
of the evolving curve.
In fact, in the recent paper \cite{finsler}, the authors proved that
solutions to the strictly parabolic PDE
\begin{equation} \label{eq:Hxt}
H(x,x_\rho) x_t = [\Phi_p(x, x_\rho)]_\rho - \Phi_z(x,x_\rho) 
\end{equation}
also satisfy \eqref{eq:acsfI}. 
Here 
\begin{equation} \label{eq:Phi}
\Phi(z,p) = \tfrac12 \altm^2(z) \gamma^2(z,p^\perp),
\end{equation}
with $\Phi_z$ denoting the gradient of $z \mapsto \Phi(z,\cdot)$, 
and the matrix
\begin{equation*} %\label{eq:Hfinal}
H(z,p)
= \frac{\altm^2(z) \gamma(z,p^\perp)}{| \gamma_p(z,p^\perp) |^2}
\begin{pmatrix} 
\gamma(z,p^\perp) & \gamma_p(z,p^\perp) \cdot p \\
- \gamma_p(z,p^\perp) \cdot p & \gamma(z, p^\perp) 
\end{pmatrix} \quad \forall\ z \in \Omega,\ p \in \bR^2\setminus\{0\}
\end{equation*}
is positive definite in $\Omega \times (\bR^2\setminus\{0\})$ with
\begin{equation} \label{eq:Hposdef}
H(z,p) \xi \cdot \xi = \frac{\altm^2(z) \gamma^2(z,p^\perp)}
{| \gamma_p(z,p^\perp) |^2} |\xi|^2 \quad \forall\
z \in \Omega,\ p \in \bR^2 \setminus \{ 0 \},\ \xi \in \bR^2.
\end{equation}
The weak formulation of \eqref{eq:Hxt} is %given by
obtained by multiplying it with a test function, integrating over $I$
and performing one integration by parts. It reads as follows.
Given $x_0 : I \to \Omega$, find $x:I \times [0,T] \to \Omega$ such 
that $x(\cdot,0)= x_0$ and, for $t \in (0,T]$,
\begin{equation} \label{eq:Hxtweak}
\int_I H(x,x_\rho) x_t \cdot \eta \drho
+ \int_I \Phi_p(x,x_\rho) \cdot \eta_\rho \drho + \int_I \Phi_z(x,x_\rho) \cdot \eta \drho = 0 \quad \forall\ \eta \in [H^1(I)]^2.
\end{equation}
For a continuous-in-time semidiscrete finite element approximation 
of \eqref{eq:Hxtweak} using piecewise linear elements the authors were then
able to prove an optimal $H^1$--error bound, see \cite[Theorem~4.1]{finsler}. 

In this paper we propose and analyse a fully discrete finite element
approximation of \eqref{eq:Hxtweak}. 
The scheme, which will be introduced in Section~\ref{sec:fea}, is
nonlinear and uses both explicit and implicit approximations in
$\Phi_p(x,x_\rho)$ and $\Phi_z(x,x_\rho)$ that
are chosen in such a way as to yield unconditional stability. 
Here the discrete stability bound
will mimic the natural estimate
\begin{equation*} %\label{eq:stab}
\ddt \int_I \Phi( x,  x_\rho) \drho 
= - \int_I  H(x,x_\rho) x_t \cdot x_t \drho  \leq 0,
\end{equation*}
which follows from choosing $\eta = x_t$ in \eqref{eq:Hxtweak}. 
Furthermore, we prove the existence of a unique solution under a
suitable CFL condition.
Then in Section~\ref{sec:nr} we present some numerical simulations,
demonstrating the practicality of the method, 
as well as the good properties with respect to stability and 
the distribution of vertices.
Let us finally mention that
alternative numerical approximations of anisotropic variants of curve
shortening flow, 
which are based on a parametric description of the moving curve,
have also been considered in 
\cite{BenesM98,Dziuk99,MikulaS01,MikulaS04,MikulaS04a,HausserV06,Pozzi07,%
triplejANI,curves3d,fdfi,hypbol}. 

\section{Finite element approximation} \label{sec:fea}

Let $[0,1]=\bigcup_{j=1}^J I_j$, $J\geq3$, be a
decomposition of $[0,1]$ into the intervals %given by the nodes $q_j$,
$I_j=[q_{j-1},q_j]$, where, for simplicity, 
$q_j = jh$, $j=0,\ldots,J$, with $h = \frac 1J$.
Within $I$ we identify $q_J=1$ with $q_0=0$ and 
define the finite element space
$\Vh = \{\chi \in C^0(I,\bR^2) : \chi\!\mid_{I_j} 
\text{ is affine},\ j=1,\ldots, J\}$.
For two piecewise continuous functions, with possible jumps at the 
nodes $\{q_j\}_{j=1}^J$, we define the mass lumped $L^2$--inner product 
\begin{equation}
( u, v )^h = \tfrac12\sum_{j=1}^J h_j
\left[(u\cdot v)(q_j^-) + (u \cdot v)(q_{j-1}^+)\right],
\label{eq:ip0}
\end{equation}
where $(u \cdot v)(q_j^\pm)=\underset{\delta\searrow 0}{\lim}\ 
(u \cdot v)(q_j\pm\delta)$. We define the associated norm on $\Vh$ via
$\| u \|_h^2 = (u, u)^h$.

In order to discretize in time, let $t_m=m \Delta t$, $m=0,\ldots,M$, 
with the uniform time step $\Delta t = \frac TM >0$. 
On recalling \eqref{eq:Phi}, we assume that 
there exists a splitting $\Phi = \Phi^+ + \Phi^-$ such that $\Phi^\pm \in C^1(\Omega \times \bR^2)$ and
$z \mapsto \pm \Phi^\pm(z,p)$ are convex in $\Omega$ for all $p\in\bR^2$. Furthermore, we assume that
for every compact set $K \subset \Omega$ there exists $L_K \geq 0$ such that
\begin{equation} \label{eq:lip1}
| \Phi_z^\pm(z,p) - \Phi_z^\pm(z,q) | \leq L_K (|p| + |q|) | p-q| \quad \text{ for all } z \in K,\ p,q \in \bR^2.
\end{equation}

Then our finite element
scheme is defined as follows.
Given $x^m_h \in \Vh$ with $\Gamma^m_h:=x^m_h(I) \subset \Omega$, 
for $m=0,\ldots,M-1$, find $x^{m+1}_h \in \Vh$ such that
$\Gamma^{m+1}_h \subset \Omega$ and 
\begin{align}\label{eq:fea}
& 
\frac1{\Delta t} \left(H(x^m_h,x^m_{h,\rho}) (x^{m+1}_{h} - x^m_h), \eta_h 
\right)^h
+ \left(\Phi_p(x^m_h,x^{m+1}_{h,\rho}) , \eta_{h,\rho} \right)^h
\nonumber \\ & \qquad
+ \left( \Phi^+_z(x^{m+1}_h,x^{m+1}_{h,\rho}) 
+ \Phi^-_z(x^m_h,x^{m+1}_{h,\rho}) ,  \eta_h \right)^h =0
\qquad \forall\ \eta_h \in \Vh.
\end{align}
The convex/concave splitting employed for the implicit/explicit
approximation of $\Phi_z(\cdot, x^{m+1}_{h,\rho})$ in \eqref{eq:fea}, 
is by now standard practice in the
numerical analysis community. This technique goes back to \cite{ElliottS93},
see also \cite{BarrettB97,vch,hypbol,finsler} for subsequent applications of
such splittings. It leads to an unconditionally stable approximation, as we
show in our first result.

\begin{theorem} %\label{thm:stab}
Any solution of \eqref{eq:fea} satisfies the energy estimate 
\begin{align} \label{eq:fdstab}
& \left( \Phi( x^{m+1}_h,  x^{m+1}_{h,\rho}),1 \right)^h 
+ \frac1{\Delta t} 
\left( H(x^m_h,x^m_{h,\rho}) (x^{m+1}_h - x^m_h), x^{m+1}_h - x^m_h\right)^h 
\nonumber \\ & \qquad
\leq \left(\Phi( x^m_h, x^m_{h,\rho}) ,1 \right)^h .
\end{align}
\end{theorem}
\begin{proof}
{From} the convexity properties of $\Phi$ and $\pm\Phi^\pm$ we infer that
\begin{align} \label{eq:Phiconvex}
\left(\Phi_p(x^m_h, x^m_{h,\rho}+\eta_{h,\rho}), \eta_{h,\rho} \right)^h 
&\geq \bigl( \Phi(x^m_h,x^m_{h,\rho}+\eta_{h,\rho}),1 \bigr)^h
- \bigl( \Phi(x^m_h,x^m_{h,\rho}),1 \bigr)^h, \nonumber \\
\left( \Phi^+_z(x^m_h+\eta_h,x^m_{h,\rho}+\eta_{h,\rho}),\eta_h \right)^h 
& \geq \bigl( \Phi^+(x^m_h+\eta_h,x^m_{h,\rho}+\eta_{h,\rho}),1 \bigr)^h 
\nonumber \\ & \qquad
- \bigl( \Phi^+(x^m_h,x^m_{h,\rho}+\eta_{h,\rho}),1 \bigr)^h, \nonumber \\
\left( \Phi^-_z(x^m_h,x^m_{h,\rho}+\eta_{h,\rho}), \eta_h \right)^h 
& \geq \left( \Phi^-(x^m_h+ \eta_h,x^m_{h,\rho}+\eta_{h,\rho}), 1 \right)^h 
\nonumber \\ & \qquad
- \left( \Phi^-(x^m_h,x^m_{h,\rho}+\eta_{h,\rho}),1 \right)^h,
\end{align}
for all $\eta_h \in \Vh$. Choosing $\eta_h = x^{m+1}_h - x^m_h$ in
\eqref{eq:fea} and applying \eqref{eq:Phiconvex} yields the bound
\eqref{eq:fdstab}. 
\end{proof}

We note that the fully discrete finite element approximation \eqref{eq:fea} 
can be seen as a generalization of two fully discrete schemes introduced by the
authors in \cite{finsler}. In particular, 
in the special case of a spatially homogeneous anisotropy, recall
\eqref{eq:classicalacsf}, 
the scheme
\eqref{eq:fea} reduces to \cite[(5.4)]{finsler}. Similarly, in the case when
\eqref{eq:acsfI} models geodesic flow in a Riemannian manifold, 
the approximation \cite[(5.10)]{finsler}
is a special case of the scheme \eqref{eq:fea}.
Moreover, we remark that the
nonlinear systems of equations arising from \eqref{eq:fea} 
can be solved with a Newton method or with a Picard-type iteration. 
In our experience, in general, in practice these solution methods converge
within a few iterations.

Let us next address the existence and uniqueness for the nonlinear system \eqref{eq:fea}. 
We assume that $x^m_h \in \Vh$ is given with $x^m_{h, \rho} \neq 0$ in $I$ and $\Gamma^m_h \subset \Omega$.
There exists $R>0$ such that
$K_0 := \{ z \in \bR^2 : \dist(z, \Gamma^m_h) \leq R \} \subset \Omega$.
Before we present our main theorem, we collect the following auxiliary results.
\begin{lemma}
There exists a constant $C_0>0$ depending on $x^m_h$ such that
\begin{subequations} %\label{eq:estPhiall}
\begin{alignat}{2}
\Phi(z,p) & \geq C_0 |p|^2 \qquad &&\forall\ p \in \bR^2, z \in K_0,
\label{eq:estPhi} \\
\bigl( \Phi_p(z,q) - \Phi_p(z,p) \bigr) \cdot (q -p) & \geq 
C_0 |  q -  p |^2 \qquad && \forall\ p,q \in \bR^2, z \in K_0, \label{eq:ellip1} \\
H(x^m_h,x^m_{h,\rho}) \xi \cdot \xi & \geq C_0 |\xi|^2 \qquad &&
\forall\ \xi \in \bR^2 \quad\text{in } I. \label{eq:estH}
\end{alignat}
\end{subequations}

\end{lemma}
\begin{proof} 
The bound \eqref{eq:estPhi} follows from \eqref{eq:abshom} 
and $\Phi(z,p) >0$ for $z \in K_0$ and $|p|=1$. Similarly,
we have from \eqref{eq:Hposdef} and $\min_I |x^m_{h,\rho}| >0$ that
\eqref{eq:estH} holds. 

It remains to show \eqref{eq:ellip1}. 
Since $a\geq a_0>0$ in $K_0$, it is sufficient to carry out the proof for 
$\Phi(z,p)=\frac12 \gamma^2(z,p)$. Note that in this case $\Phi \in C^1(\Omega \times \bR^2, \bRgeq) \cap C^2(\Omega \times (\bR^2 \setminus \{ 0 \}, \bRplus)$.
Furthermore, according to \cite[Remark 1.7.5]{Giga06},
$\gamma^2_{pp}(z,p) := [\gamma^2]_{pp}(z,p)$ is positive definite for $p \neq 0$. In particular, there exists $c_0>0$ such that 
\[
\gamma^2_{pp}(z,p) q \cdot q \geq c_0 |q|^2 \qquad \text{ for all } z \in K_0,\ p, q \in \bR^2, | p | =1.
\]
Observing that $\Phi_{pp}(z,p)= A \gamma^2_{pp}(z,p^\perp) A^T$ with 
$A=\binom{0\ -1}{1\ \phantom{-}0}$, 
we infer
that 
\begin{equation} \label{eq:Phiell} 
\Phi_{pp}(z,p) q \cdot q =  \gamma^2_{pp}(z,p^\perp)A^T q \cdot A^T q
\geq c_0 | A^T q |^2 = c_0 | q|^2
\end{equation}
for all $z \in K_0$, $p,q \in \bR^2, | p|=1$. 
Let us fix $z \in K_0$ and $p,q \in \bR^2$. We distinguish two cases: \\
{\it Case 1:} $s q + (1-s) p \neq 0$ for all $s \in [0,1]$. Then  
\[
\bigl( \Phi_p(z,q) - \Phi_p(z,p) \bigr) \cdot (q -p)  = \int_0^1 \Phi_{pp}(z,sq +(1-s) p) (q-p) \cdot (q-p) \ds \geq c_0 | q-p|^2,
\]
using the fact that $p \mapsto \Phi_{pp}(z,p)$ is 0-homogeneous and 
\eqref{eq:Phiell}. \\
{\it Case 2:} There exists $s \in [0,1]$ with $sq + (1-s) p =0$. 
We may assume that $s \in [0,1)$, so that
$p= - \frac{s}{1-s} q$. 
As $\Phi(z,\lambda q) = \lambda^2 \Phi(z,q)$ for $\lambda \in \bR$,
recall \eqref{eq:abshom}, we have that
$\Phi_p(z,\lambda q) = \lambda \Phi_p(z,q)$ and 
$\Phi_p(z,q) \cdot q = 2 \Phi(z,q)$.
Hence we obtain that
\begin{align*}
\bigl( \Phi_p(z,q) - \Phi_p(z,p) \bigr) \cdot (q -p) & = \bigl( \Phi_p(z,q)- \Phi_p(z,- \tfrac{s}{1-s} q \bigr)
\cdot \bigl( q +  \tfrac{s}{1-s} q \bigr) \\
& = \bigl( 1 + \tfrac{s}{1-s} \bigr)^2 \Phi_p(z,q) \cdot q 
= 2 \bigl( 1 + \tfrac{s}{1-s} \bigr)^2 \Phi(z,q) \\
& \geq 2 C_0 %\min_{z \in K_0, | \nu |=1} \Phi(z,\nu) 
\bigl( 1 + \tfrac{s}{1-s} \bigr)^2 |q|^2 = 2 C_0 | q-p |^2,
\end{align*}
on noting \eqref{eq:estPhi}.
\end{proof}

We are now in a position to prove our main result.

\begin{theorem} %\label{thm:ex}
There exists $\delta>0$ such that for $\Delta t \leq \delta h$
there is a unique element $x^{m+1}_h \in \Vh$ with 
$\Gamma^{m+1}_h \subset K_0$ which solves \eqref{eq:fea}. The constant
$\delta$ depends on $R,C_0, L_{K_0}$ and $\Phi(x^m_h,x^m_{h,\rho})$.
\end{theorem}
\begin{proof}
We denote by $\{\varphi_j\}_{j=1}^{2J}$ the basis of $\Vh$ satisfying 
$\varphi_j(q_k) = \delta_{jk} e_1$ and $\varphi_{j+J}(q_k)
= \delta_{jk} e_2$ for $1 \leq j,k \leq J$, and associate with every $\alpha \in \mathbb R^{2J}$ the element
$v_\alpha = \sum_{j=1}^{2J} \alpha_j \varphi_j \in \Vh$, 
so that $v_\alpha(q_j)= \binom{\alpha_j}{\alpha_{j+J}}$. 
We have for $| \alpha | \leq R$ and $\rho \in I$ that
\[
 \dist((x^m_h+ v_\alpha)(\rho),\Gamma^m_h) \leq \Vert v_\alpha \Vert_\infty \leq  | \alpha | \leq  R,
\]
so that $(x^m_h+ v_\alpha)(\rho) \in K_0$.  
Let us next define the continuous map $F:\overline{B_R(0)} \times [0,1] \to \bR^{2J}$ via
\begin{align*}
[F(\alpha,\lambda)]_i & = 
\frac1{\Delta t}
\left(H(x^m_h,x^m_{h,\rho}) v_\alpha, \varphi_i \right)^h
+ \lambda \left(\Phi_p(x^m_h,x^m_{h,\rho}+ v_{\alpha,\rho}) , \varphi_{i,\rho} \right)^h
\nonumber \\ & \qquad
+ \lambda \left( \Phi^+_z(x^m_h+v_\alpha,x^m_{h,\rho}+v_{\alpha,\rho}) 
+ \Phi^-_z(x^m_h,x^m_{h,\rho}+v_{\alpha,\rho}) ,  \varphi_i \right)^h .
\end{align*}
In what follows we make use of standard results for the 
Brouwer degree \linebreak $\deg (f, B_R(0), 0)$
of a continuous function $f:\overline{B_R(0)} \to \bR^{2J}$, 
if $0 \not\in f(\partial B_R(0))$,
see \cite[Chapter~1]{Deimling85}. % \cite[\S1.6A]{Berger77}
Clearly, the mapping $F(\cdot,0) =: A$ is linear with
\[
[A \alpha]_i= \frac1{\Delta t} \sum_{j=1}^{2J} \alpha_j \left(H(x^m_h,x^m_{h,\rho}) \varphi_j, \varphi_i \right)^h, \quad
i=1,\ldots,2J,
\]
and invertible in view of \eqref{eq:estH}.
Hence it follows from \cite[Theorem~1.1]{Deimling85} that 
\begin{equation} \label{eq:deg0}
\deg(F(\cdot,0),B_R(0),0) = \deg(A,B_R(0),0) = \sgn \det A \neq 0.
\end{equation}
Next, it holds for $\lambda \in [0,1]$ that 
\begin{align}
F(\alpha,\lambda) \cdot \alpha 
& =  \frac1{\Delta t} \left(H(x^m_h,x^m_{h,\rho}) v_\alpha, v_\alpha \right)^h
+ \lambda \left(\Phi_p(x^m_h, x^m_{h,\rho}+v_{\alpha,\rho}), v_{\alpha,\rho} \right)^h
\nonumber \\ & \qquad
+ \lambda \left( \Phi^+_z(x^m_h+v_\alpha,x^m_{h,\rho}+v_{\alpha,\rho}) 
+ \Phi^-_z(x^m_h,x^m_{h,\rho}+v_{\alpha,\rho}), v_\alpha \right)^h. \label{eq:stabest}
\end{align}
Inserting \eqref{eq:Phiconvex} with $\eta_h = v_\alpha$ 
into \eqref{eq:stabest} yields
\begin{align*}
F(\alpha,\lambda) \cdot \alpha & \geq \frac1{\Delta t} \left(H(x^m_h,x^m_{h,\rho}) v_\alpha, v_\alpha \right)^h \nonumber \\
& \qquad 
+ \lambda \bigl( \Phi(x^m_h+v_\alpha,x^m_{h,\rho}+v_{\alpha,\rho}),1 \bigr)^h 
- \lambda \bigl( \Phi(x^m_h,x^m_{h,\rho}), 1 \bigr)^h. % \label{eq:stabest1}
\end{align*}
If we combine this estimate with \eqref{eq:estH} we finally obtain for 
$| \alpha | =R$ that
\begin{align*}
F(\alpha,\lambda) \cdot \alpha & \geq 
 \frac{C_0}{\Delta t} \Vert v_\alpha \Vert_h^2  - \bigl( \Phi(x^m_h,x^m_{h,\rho}),1 \bigr)^h \\
 & = \frac{C_0}{\Delta t} h | \alpha |^2 
 - \bigl( \Phi(x^m_h,x^m_{h,\rho}),1 \bigr)^h   \geq \frac{C_0 R^2}{\delta}  
 - \bigl( \Phi(x^m_h,x^m_{h,\rho}),1 \bigr)^h
\end{align*}
using \eqref{eq:ip0} and the relation $\Delta t \leq \delta h$. By choosing $\delta$  sufficiently small we obtain $F(\alpha,\lambda) \cdot \alpha>0$ and therefore  
 $F(\alpha,\lambda) \neq 0$ for all 
$(\alpha,\lambda) \in \partial B_R(0) \times [0,1]$. 
Using the homotopy invariance of the
Brouwer degree together with \eqref{eq:deg0} we deduce that 
$\deg(F(\cdot,1),B_R(0),0) = \deg(F(\cdot,0),B_R(0),0) \neq 0$, 
so that the existence property of the degree shows that there is 
$\alpha \in B_R(0)$ such that
$F(\alpha,1) =0$. Clearly, $x^{m+1}_h:= x^m_h+v_\alpha$ is then a solution of \eqref{eq:fea}.

In order to prove uniqueness, suppose that $x^{m+1}_h, \tilde x^{m+1}_h \in \Vh$ are two solutions of \eqref{eq:fea} with
$\Gamma^{m+1}_h, \tilde \Gamma^{m+1}_h \subset K_0$. To begin, we infer from \eqref{eq:fdstab} and \eqref{eq:estPhi} that
\begin{equation} \label{eq:firstderiv}
C_0 \Vert x^{m+1}_{h,\rho} \Vert_h^2 \leq \bigl( \Phi(x^m_h,x^m_{h,\rho}),1 \bigr)^h, \;
 C_0 \Vert \tilde x^{m+1}_{h,\rho} \Vert_h^2 \leq \bigl( \Phi(x^m_h,x^m_{h,\rho}),1 \bigr)^h.
\end{equation}
We have that 
\begin{align*}
&
\frac1{\Delta t} \left(H(x^m_h,x^m_{h,\rho}) (x^{m+1}_{h} - \tilde x^{m+1}_h), \eta_h 
\right)^h 
+ \left(\Phi_p(x^m_h,x^{m+1}_{h,\rho}) - \Phi_p(x^m_h,\tilde x^{m+1}_{h,\rho}), 
\eta_{h,\rho} \right)^h \\ & \qquad
+ \left( \Phi^+_z(x^{m+1}_h,x^{m+1}_{h,\rho}) - \Phi^+_z(\tilde x^{m+1}_h,x^{m+1}_{h,\rho}) 
,  \eta_h \right)^h \\ & \quad 
= \left( \Phi^+_z(\tilde x^{m+1}_h, \tilde x^{m+1}_{h,\rho}) - \Phi^+_z(\tilde x^{m+1}_h,x^{m+1}_{h,\rho}) 
,  \eta_h \right)^h \\ & \qquad 
+ \left( \Phi^-_z(x^m_h,\tilde x^{m+1}_{h,\rho}) -  \Phi^-_z(x^m_h,x^{m+1}_{h,\rho})
,  \eta_h \right)^h
\end{align*}
for all $\eta_h \in \Vh$.
Choosing $\eta = x^{m+1}_h - \tilde x^{m+1}_h$ we deduce with the help of \eqref{eq:estH}, \eqref{eq:ellip1},
\eqref{eq:lip1},  \eqref{eq:firstderiv} and the convexity of $z \mapsto \Phi^+(z,p)$ that
\begin{align*}
&
\frac{C_0}{\Delta t} \| x^{m+1}_{h} - \tilde x^{m+1}_h \|_h^2 
+ C_0 \| x^{m+1}_{h,\rho} - \tilde x^{m+1}_{h,\rho} \|_h^2 \\ & \qquad
\leq 2 L \bigl( \Vert x^{m+1}_{h,\rho} \Vert_h + \Vert \tilde x^{m+1}_{h,\rho} \Vert_h \bigr) \Vert x^{m+1}_{h,\rho} -
\tilde x^{m+1}_{h,\rho} \Vert_h \Vert x^{m+1}_h - \tilde x^{m+1}_h \Vert_\infty \\ & \qquad
\leq 4 L C_0^{-\frac{1}{2}} h^{-\frac{1}{2}} \sqrt{\bigl( \Phi(x^m_h,x^m_{h,\rho}),1 \bigr)^h} \Vert x^{m+1}_{h,\rho} - \tilde x^{m+1}_{h,\rho} \Vert_h  \| x^{m+1}_{h} - \tilde x^{m+1}_h \|_h \\ & \qquad
\leq C_0 \| x^{m+1}_{h,\rho} - \tilde x^{m+1}_{h,\rho} \|_h^2 + 4 L^2 \bigl( \Phi(x^m_h,x^m_{h,\rho}),1 \bigr)^h C_0^{-2}
 h^{-1} \| x^{m+1}_{h} - \tilde x^{m+1}_h \|_h^2,
\end{align*}
with $L= L_{K_0}$.
By choosing $\delta>0$ so small that $4 L^2  \bigl( \Phi(x^m_h,x^m_{h,\rho}),1 \bigr)^h \, \delta < C_0^3$ we deduce that $x^{m+1}_h = \tilde x^{m+1}_h$.
\end{proof}

\section{Numerical results} \label{sec:nr}

For all our numerical simulations %we use a uniform partitioning of $[0,1]$,
we use $J=256$ and $\Delta t = 10^{-4}$.
On recalling (\ref{eq:aniso}), for $\chi_h \in \Vh$ we define the discrete
energy
\begin{equation*} % \label{eq:Lgh}
\mathcal{E}^{h}(\chi_h) = 
\left( \gamma(\chi_h, \chi_{h,\rho}^\perp), \altm(\chi_h)\right)^{h} .
\end{equation*}
We also consider the ratio
\begin{equation*} %\label{eq:ratio}
\ratio^m = \dfrac{\max_{j=1,\ldots,J} |x_h^m(q_j) - x_h^m(q_{j-1})|}
{\min_{j=1,\ldots, J} |x_h^m(q_j) - x_h^m(q_{j-1})|}
\end{equation*}
between the longest and shortest element of $\Gamma^m_h = x_h^m(I)$, 
and are often interested in the evolution of this ratio over time.
We stress that no redistribution of vertices was necessary during any of our
numerical simulations.
We remark that a convergence experiment for \eqref{eq:fea},
for an anisotropy of the form
$\gamma(z,p) = \sqrt{p_1^2 + \delta^2 p_2^2}$ with $\delta > 0$, which
confirms the theoretically obtained optimal $H^{1}$--error bound, 
can be found in \cite[\S6.1]{finsler}.

\subsection{The spatially homogeneous case}

In the case that
\begin{equation*} %\label{eq:homogeneous}
\gamma(z,p)= \gamma_0(p) \quad \text{and} \quad \altm(z)= 1
\qquad \forall\ z \in \Omega = \bR^2,
\end{equation*}
for an anisotropy function
$\gamma_0 \in C^2(\bR^2 \setminus\{0\}, \bRplus)$,
the flow \eqref{eq:acsfI} reduces to classical anisotropic curvature flow,
\eqref{eq:classicalacsf}. 
Most existing approaches for
the numerical approximation of 
anisotropic curve shortening flow deal with this simpler case, see e.g.\ 
\cite{Dziuk99,Pozzi07,triplejANI,fdfi}.

For our first experiment we choose the anisotropy from \cite[(7.1)]{Dziuk99},
with 
\begin{equation} \label{eq:gammaD99b}
\gamma_0( p) = | p| (1 + \delta \cos(k\theta(p))),
\quad  p = | p| \binom{\cos\theta(p)}{\sin\theta(p)},\quad 
k = 6 ,\ \delta = 0.028,
\end{equation}
and as initial curve use the one given in
\cite[p.\ 1494]{MikulaS01}, i.e.\ we let
\begin{equation} \label{eq:mikula}
 x(\rho,0)=\binom{\cos{u(\rho)}}
  {\tfrac12\sin{u(\rho)} + \sin{(\cos{u(\rho)})} + \sin{u(\rho)}\,
  [\tfrac15+\sin{u(\rho)}\,\sin^2{u(3\rho)}]}, 
\end{equation}
where $u(\rho) = 2\pi\rho$. The evolution is shown in 
Figure~\ref{fig:mikula}, where we can observe that the shrinking curve 
becomes convex, with its form soon approaching a scaled Wulff shape of the 
six-fold symmetric anisotropy \eqref{eq:gammaD99b}. 
In addition, we once again note that the discrete energy $\mathcal{E}^h$ is
monotonically decreasing, while the tangential motion induced by \eqref{eq:Hxt} leads to only a moderate initial increase in $\ratio^m$, before it decreases 
towards the end. 
\begin{figure}
\center
\includegraphics[angle=-90,width=0.3\textwidth]{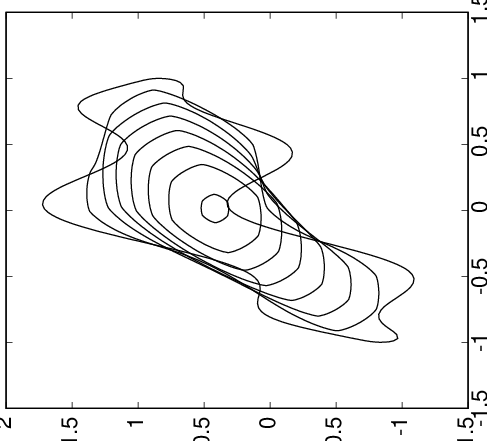}
\includegraphics[angle=-90,width=0.34\textwidth]{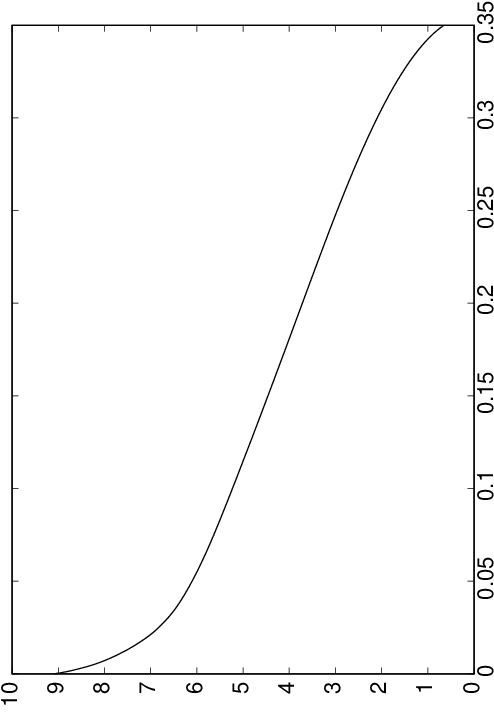}
\includegraphics[angle=-90,width=0.34\textwidth]{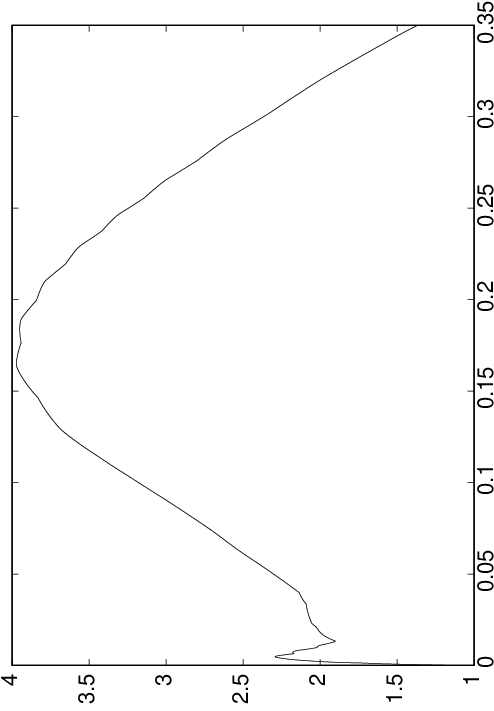}
\caption{Solution at times $t=0,0.05,\ldots,0.35$.
We also show a plot of the discrete energy $\mathcal{E}^h(x_h^m)$ (middle)
and of the ratio $\ratio^m$ over time (right).} 
\label{fig:mikula}
\end{figure}
With our next experiment we would like to demonstrate that our scheme can
easily be extended to situations where a forcing term appears in the flow,
e.g.\
\begin{equation} \label{eq:Fflow}
\mathcal V_\gamma = \varkappa_\gamma + f(\cdot,\nu)
\quad \text{ on } \Gamma(t),
\end{equation}
in place of \eqref{eq:anisolaw}, where $f : \Omega \times \bR^2 \to \bR$. 
This leads to the additional term 
$\int_I f(x,\nu) H(x, x_\rho) \nu \cdot \eta\drho$ on the right hand side of
\eqref{eq:Hxtweak}, and analogously to the additional term 
$( f(x^m_h,\frac{(x^{m}_{h,\rho})^\perp}{|x^m_{h,\rho}|}) 
H(x^m_\rho, x^m_{h,\rho}) \frac{(x^{m}_{h,\rho})^\perp}{|x^m_{h,\rho}|}, 
\eta_h )^h$ on the right hand side of \eqref{eq:fea}. In our numerical
experiments we choose
\[
f(z,\nu) = f_0 \in \bR,
\]
so that \eqref{eq:Fflow} overall reduces to
$\frac{1}{\gamma_0(\nu)} \mathcal V = \varkappa_{\gamma_0} + f_0$,
compare with \eqref{eq:classicalacsf}. Starting this flow from the same
initial data \eqref{eq:mikula}, but now with the forcing $f_0 = 1.15$,
leads to an expanding curve that, upon an appropriate rescaling, 
approaches the boundary of the Wulff shape, see Figure~\ref{fig:mikulaF}. 
What is particularly interesting in the observed evolution is that the ratio
$\ratio^m$ appears to tend towards unity, indicating an asymptotic
equidistribution of the vertices on the polygonal curve.
\begin{figure}
\center
\includegraphics[angle=-90,width=0.3\textwidth]{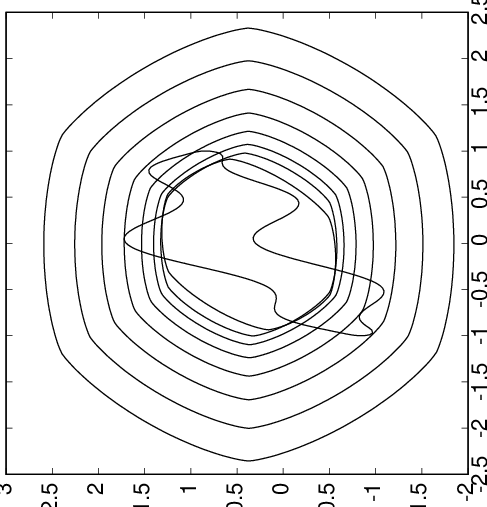}
\includegraphics[angle=-90,width=0.34\textwidth]{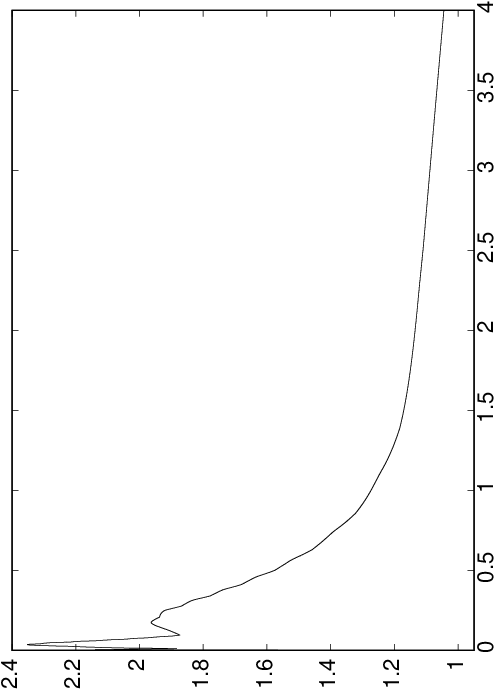}
\caption{Solution at times $t=0,0.5,\ldots,4$.
We also show a plot of 
the ratio $\ratio^m$ over time.} % (right).}
\label{fig:mikulaF}
\end{figure}

\subsection{Geodesic curvature flow in Riemannian manifolds}
\label{sec:geo}
Let $(\mathcal M,g)$ be a 
two-dimensional Riemannian manifold, with local parameterization 
$F: \Omega \to \mathcal M$ and corresponding basis 
$\{ \partial_1,\partial_2 \}$ of the tangent space. We define 
\begin{equation*} %\label{eq:defG}
\gamma(z,p)= \sqrt{G^{-1}(z) p \cdot p} \quad \text{and} \quad 
\altm(z)=\sqrt{\det G(z)}.
\end{equation*}
Here, $G(z)=(g_{ij}(z))_{i,j=1}^2$, with 
$g_{ij}(z)=g_{F(z)}(\partial_i,\partial_j)$, $z \in \Omega$.
Then \eqref{eq:aniso} reduces to 
$\mathcal{E}(\Gamma)=\int_{\Gamma} \sqrt{G \tau \cdot \tau} \dH1$, 
the Riemannian length of the curve $\tilde \Gamma = 
F(\Gamma) \subset \mathcal M$, and it can be shown that
\eqref{eq:anisolaw} is now equivalent to geodesic curvature flow in $\mathcal M$,
see \cite[Appendix~B]{finsler} for details. 
Furthermore, in \cite[\S5.2]{finsler} a possible construction of the splitting
$\Phi_z^\pm$ is given, with the help of which the scheme \eqref{eq:fea} reduces
to (5.10) in \cite{finsler}. 

As an example %for geodesic curvature flow in a Riemannian manifold 
we consider the case when $(\mathcal M,g)$ is a hypersurface in the
Euclidean space $\bR^3$. Assuming that $\mathcal M$ can be written as a graph,
we let
\begin{equation*} %\label{eq:Fvarphi}
F(z) = (z_1, z_2, \varphi(z))^T,\quad \varphi \in C^3(\Omega).
\end{equation*} 
The induced matrix $G$ is then given by 
$G(z) = \Id + \nabla \varphi(z) \otimes \nabla \varphi(z)$,
and we have that $\Phi(z,p) = \frac12 G(z) p \cdot p$.
For the splitting $\Phi = \Phi^+ + \Phi^-$ it is natural to 
let $\Phi^+(z,p) = \frac12 G_+(z) p \cdot p$, where
$G_+(z)=G(z) + c_\varphi |z|^2 \Id$ and 
$c_\varphi \in \bRgeq$ is chosen sufficiently large. 
In our computation
we observe a monotonically decreasing discrete energy when choosing
$c_\varphi=0$, and so we let $\Phi^+ = \Phi$.

For our numerical simulation, following \cite{WuT10}, we define
a surface with three ``mountains'' via
\begin{equation} \label{eq:varphi}
\varphi(z) = \sum_{i=1}^3 \lambda_i \psi(2|z - \mu_i|^2)
\text{ with } \Omega = \bR^2,
\quad \text{where } \
\psi(s) = \begin{cases}
 e^{- \frac1{1 - s}} & s < 1,\\
 0 & s \geq 1,
\end{cases} 
\end{equation}
and where $\mu_1 = 0$, $\mu_2 = \tbinom20$, $\mu_3 = \tbinom1{\sqrt{3}}$
and $(\lambda_1, \lambda_2, \lambda_3) = (1,3,4)$. 
On letting the initial polygonal curve be defined by
an equidistributed approximation of a circle of radius 2 in $\Omega$, 
centred at $\frac13 \sum_{i=1}^3 \mu_i$, 
we show the evolution for geodesic curvature flow on 
the defined hypersurface in Figure~\ref{fig:riemmount3}.
During the flow the curve tries to decrease its (Euclidean) length, while
remaining on the manifold. As the initial circle begins to shrink, the curve is
able to pass over the smallest of the three ``mountains'', but then reaches a
steady state solution encompassing the two taller mountains. Here the curve
cannot reduce its length further, because to rise higher up would yield an
increase in its length, since it needs to remain attached to the flat part of 
the surface between the two mountains.
\begin{figure}
\center
\includegraphics[trim=100 50 90 50,clip,width=0.21\textwidth]{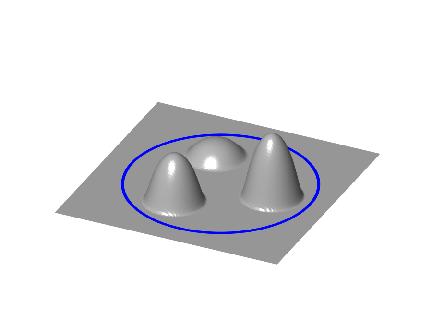}
\includegraphics[trim=100 50 90 50,clip,width=0.21\textwidth]{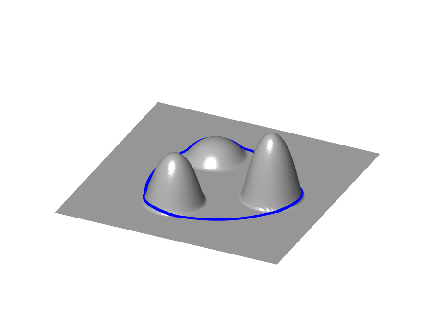}
\includegraphics[trim=100 50 90 50,clip,width=0.21\textwidth]{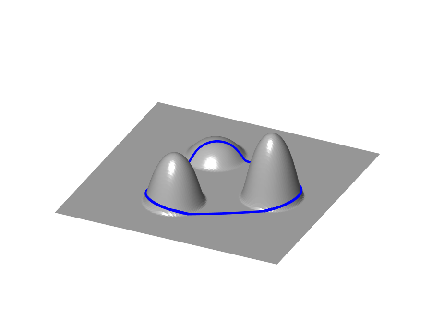}
\includegraphics[trim=100 50 90 50,clip,width=0.21\textwidth]{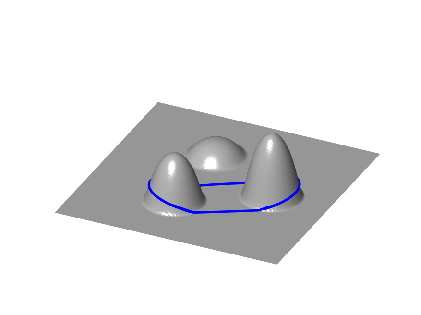} 
\includegraphics[angle=-90,width=0.34\textwidth]{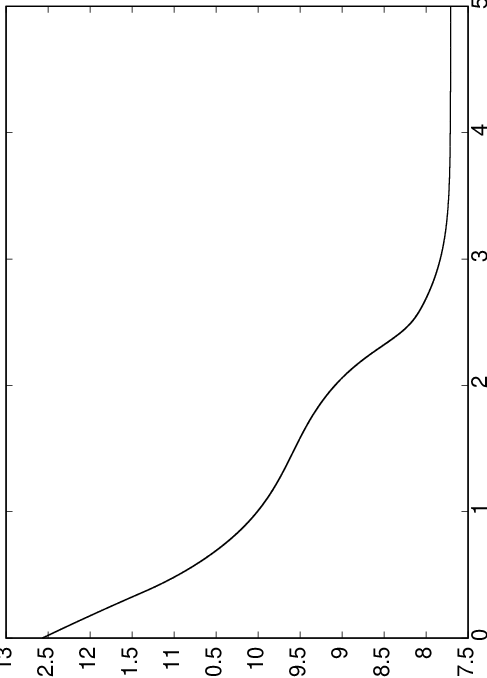}
\includegraphics[angle=-90,width=0.34\textwidth]{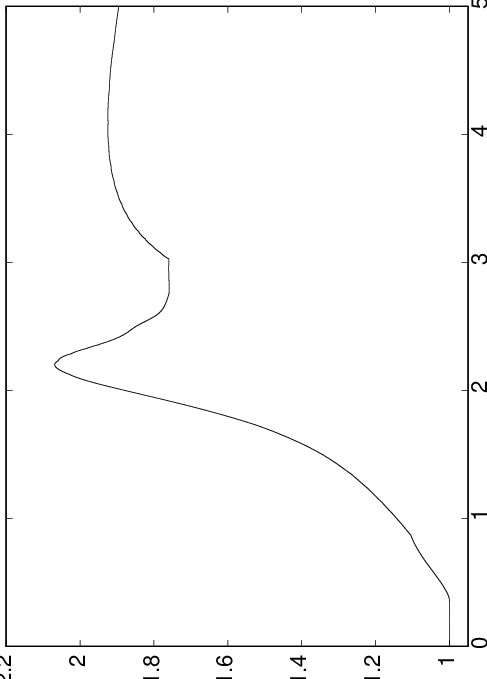}
\caption{Geodesic curvature flow on the graph defined by
\eqref{eq:varphi}.
We show the evolution of $F(x^m_h)$ on $\mathcal M$ at times $t=0, 1, 2, 4$.
Below we show a plot of the discrete energy $\mathcal{E}^h(x_h^m)$ (left)
and of the ratio $\ratio^m$ over time (right).} 
\label{fig:riemmount3}
\end{figure}%
We note that as soon as one of the larger ``mountains'' is not enclosed by the
initial curve, then the evolution is going to lead to extinction in finite
time.
We show this in Figure~\ref{fig:riemmount33}, where the initial circle is
now centred at $\binom{0}{\frac12}$. Here the curve can continually decrease its
length, until it reaches the peak of the tallest ``mountain'', where it shrinks
to a point.
\begin{figure}
\center
\includegraphics[trim=100 50 90 50,clip,width=0.21\textwidth]{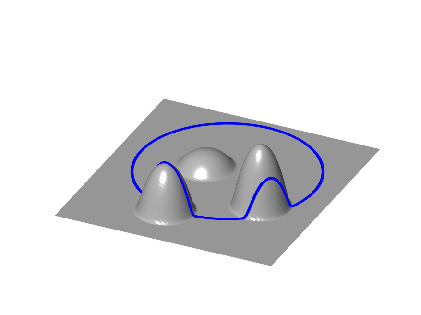}
\includegraphics[trim=100 50 90 50,clip,width=0.21\textwidth]{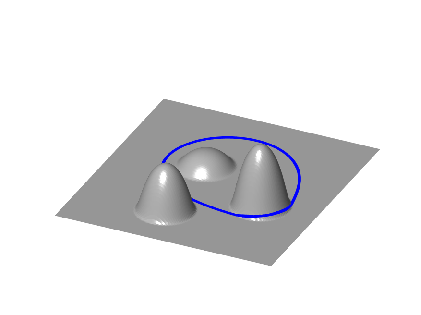}
\includegraphics[trim=100 50 90 50,clip,width=0.21\textwidth]{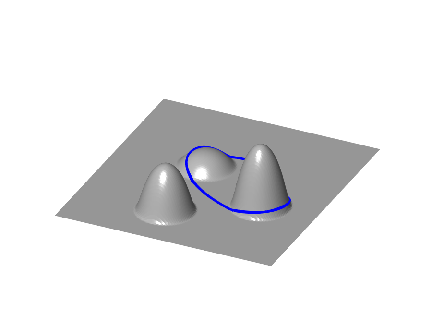}
\includegraphics[trim=100 50 90 50,clip,width=0.21\textwidth]{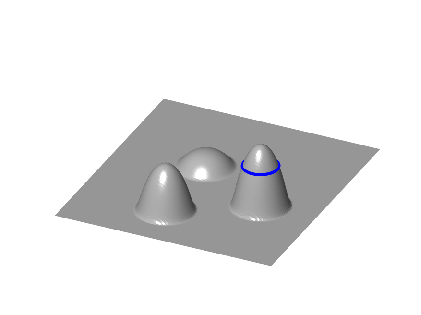} 
\caption{Geodesic curvature flow on the graph defined by
\eqref{eq:varphi}.
We show the evolution of $F(x^m_h)$ on $\mathcal M$ at times $t=0, 1, 2, 4$.
} 
\label{fig:riemmount33}
\end{figure}%

\def\soft#1{\leavevmode\setbox0=\hbox{h}\dimen7=\ht0\advance \dimen7
  by-1ex\relax\if t#1\relax\rlap{\raise.6\dimen7
  \hbox{\kern.3ex\char'47}}#1\relax\else\if T#1\relax
  \rlap{\raise.5\dimen7\hbox{\kern1.3ex\char'47}}#1\relax \else\if
  d#1\relax\rlap{\raise.5\dimen7\hbox{\kern.9ex \char'47}}#1\relax\else\if
  D#1\relax\rlap{\raise.5\dimen7 \hbox{\kern1.4ex\char'47}}#1\relax\else\if
  l#1\relax \rlap{\raise.5\dimen7\hbox{\kern.4ex\char'47}}#1\relax \else\if
  L#1\relax\rlap{\raise.5\dimen7\hbox{\kern.7ex
  \char'47}}#1\relax\else\message{accent \string\soft \space #1 not
  defined!}#1\relax\fi\fi\fi\fi\fi\fi}

\end{document}